\documentclass{birkjour}
 \newtheorem{thm}{Theorem}[section]
 
 \newtheorem{lem}[thm]{Lemma}
 
 \theoremstyle{definition}
 
 \theoremstyle{remark}
 \newtheorem{rem}[thm]{Remark}
 \newtheorem*{ex}{Example}
 \numberwithin{equation}{section}

\begin{document}

%
%
%
%
%
%
%
%
%

\title[The generalized $\partial$-complex on the Segal Bargmann space]{The generalized $\partial$-complex on the Segal -Bargmann space}

\author{Friedrich Haslinger}

 \address{  Fakult\"at  f\"ur Mathematik\\ Universit\"at Wien\\
Oskar-Morgenstern-Platz 1\\ A-1090 Wien, Austria}

\email{ friedrich.haslinger@univie.ac.at}

\thanks{Partially supported by the Austrian Science Fund (FWF) project  P28154.}

\subjclass{Primary 30H20, 32A36, 32W50; Secondary 47B38} 

\keywords{$\partial$-complex, Segal-Bargmann space}
\date{October 28, 2019}
\begin{abstract}We study certain densely defined unbounded operators on the Segal-Barg\-mann space, related to the  annihilation and creation operators of quantum mechanics.  We consider the corresponding $D$-complex and study properties of the corresponding complex Laplacian  $\tilde \Box_D = D D^* + D^* D,$ where $D$ is a differential operator of polynomial type. 	
\end{abstract}

\maketitle

\section{Introduction}

We consider the classical Segal-Bargmann space 
$$A^2(\mathbb C^n, e^{-|z|^2} ) = \{ u: \mathbb C^n \longrightarrow \mathbb C \ {\text{entire}} : \int_{\mathbb C^n} |u(z)|^2 e^{-|z|^2} \, d\lambda (z) <\infty \}$$
with inner product 
$$(u,v) =  \int_{\mathbb C^n} u(z)\, \overline{v(z)} \, e^{-|z|^2} \, d\lambda (z)$$
and 
replace a single derivative with respect to $z_j$ by a differential operator of the form $p_j(\frac{\partial}{\partial z_1}, \dots , \frac{\partial}{\partial z_n}),$ where $p_j$ is a complex polynomial on $\mathbb C^n$ (see \cite{NS1}, \cite{NS2}). We write $p_j(u)$ for 
$p_j(\frac{\partial}{\partial z_1}, \dots , \frac{\partial}{\partial z_n})u,$ where $u\in A^2(\mathbb C^n, e^{-|z|^2} ),$ and 
consider the densely defined operators
\begin{equation}\label{gendef1}
Du = \sum_{j=1}^n p_j (u)\, dz_j,
\end{equation}
where $u\in A^2(\mathbb C^n, e^{-|z|^2})$ and $p_j(\frac{\partial}{\partial z_1}, \dots , \frac{\partial}{\partial z_n})$
are polynomial differential operators with constant coefficients.
\cite{Has10}

More generally we define
\begin{equation}\label{gendef1'}
Du =  \sum_{|J|=p}\, ' \,   \sum_{k=1}^n p_k (u_J)\, dz_k \wedge dz_J,
\end{equation}
where $u=  \sum_{|J|=p}\, ' \, u_J\, dz_J$ is a $(p,0)$-form with coefficients in $A^2(\mathbb C^n, e^{-|z|^2}),$ here $J=(j_1, \dots , j_p)$ is a multiindex and $dz_J =dz_{j_1} \wedge \dots \wedge dz_{j_p}$ and the summation is taken only over increasing multiindices.

It is clear that $D^2=0$ and that we have
\begin{equation}\label{gendef2}
(Du, v) = (u, D^*v),
\end{equation}
where $u\in {\text{dom}}(D)= \{ u \in A^2_{(p,0)}(\mathbb{C}^n, e^{-|z|^2}): Du\in A^2_{(p+1,0)}(\mathbb{C}^n, e^{-|z|^2})\}$ and 
$$D^*v= \sum_{|K|=p-1}\, ' \, \sum_{j=1}^n p_j^*v_{jK}\, dz_K$$
for $v= \sum_{|J|=p}\, ' \, v_J\, dz_J$ and where $p_j^*(z_1, \dots ,z_n)$ is the polynomial $p_j$ with complex conjugate coefficients, taken as multiplication operator.

Now the corresponding $D$-complex has the form
 \begin{equation*}
  A^2_{(p-1,0)}(\mathbb{C}^n, e^{-|z|^2}) 
\underset{\underset{D^* }
\longleftarrow}{\overset{D }
{\longrightarrow}} A^2_{(p,0)}(\mathbb{C}^n, e^{-|z|^2}) \underset{\underset{D^* }
\longleftarrow}{\overset{D }
{\longrightarrow}} A^2_{(p+1,0)}(\mathbb{C}^n, e^{-|z|^2}).
\end{equation*}
Similar to the classical $\overline \partial$-complex (see \cite{Has10})  we consider the generalized box operator
$\tilde \Box_{D,p}:= D^*D+DD^*$
as a densely defined self-adjoint operator on $A_{(p,0)}^2(\mathbb C^n, e^{-|z|^2})$ with
$${\text{dom}} (\tilde\Box_{D,p})
 =\{ f\in {\text{dom}}(D) \cap {\text{dom}}(D^* ):D f \in {\text{dom}}(D^*) \ {\text{and}} \ D^* f \in {\text{dom}}(D )\},$$ see \cite{Has20}  for more details.
 
The $(p,0)$-forms with polynomial components are dense in $A_{(p,0)}^2(\mathbb C^n, e^{-|z|^2}).$ In addition we have

\begin{lem}\label{dense} The $(p,0)$-forms with polynomial components are also dense in
$  {\text{dom}}(D )\cap  {\text{dom}}(D^*)$ endowed with the graph norm 
$$u\mapsto (\|u\|^2+\|D u\|^2 + \|D^*u \|^2)^{1/2}.$$ 
\end{lem}

\begin{proof}Let $u=  \sum_{|J|=p}\, ' \, u_J\, dz_J\in   {\text{dom}}(D )\cap  {\text{dom}}(D^*)$ and consider the partial sums of the Fourier series expansions of
$$ u_J = \sum_\alpha u_{J,\alpha} \varphi_\alpha,$$
where 
\begin{equation}\label{vons}
\varphi_\alpha (z) = \frac{z^\alpha}{\sqrt{\pi^n \alpha !}}\ \ {\text{and}} \ \ \sum_\alpha |u_{J,\alpha} |^2 < \infty
\end{equation}
and $\alpha! = \alpha_1! \dots \alpha_n!.$ We have that $Du\in A_{(p+1,0)}^2(\mathbb C^n, e^{-|z|^2})$ and $D^*u \in A_{(p-1,0)}^2(\mathbb C^n, e^{-|z|^2}).$ Hence the partial sums of the Fourier series of the components of $Du$  converge to the components of $Du$ in $A_{(p+1,0)}^2(\mathbb C^n, e^{-|z|^2})$ and the partial sums of the Fourier series  of the components of $D^*u$ converge to the components of $D^*u$ in $A_{(p-1,0)}^2(\mathbb C^n, e^{-|z|^2}).$

\end{proof}

\section{The basic estimate}

We want to find conditions under which $\tilde \Box_{D,1}$ has a bounded inverse. For this purpose we have to consider the graph norm $(\|u\|^2 +\|Du\|^2+ \|D^*u\|^2)^{1/2}$ on ${\text{dom}}(D) \cap {\text{dom}}(D^* ).$ We refer to a theorem from \cite{Has20}, Theorem 5.1, here in a slightly different  improved form.

\begin{thm}\label{basic5}
Suppose that 
there exists a constant $C>0$ such that 
 \begin{equation}\label{comm1}
 \|u\|^2 \le C \sum_{j,k=1}^n ( [ p_k, p^*_j ] u_j, u_k),
 \end{equation}
for any $(1,0)$-form $u= \sum_{j=1}^n u_j dz_j $ with polynomial components.
Then 
\begin{equation}\label{basic6}
\|u\|^2 \le C ( \|Du\|^2+ \|D^*u\|^2),
\end{equation}
for any $u\in {\text{dom}}(D) \cap {\text{dom}}(D^* ).$
\end{thm}

 \begin{proof}
First we have 
 $$Du= \sum_{j<k} (p_j(u_k)-p_k(u_j))\, dz_j \wedge dz_k \ \ {\text{and}} \ \ 
 D^*u = \sum_{j=1}^n p_j^* u_j,$$
 hence
 $$ \|Du\|^2 + \|D^*u\|^2 = \int_{\mathbb C^n} \sum_{j<k} |p_k(u_j)-p_j(u_k)|^2 \, e^{-|z|^2} \,d\lambda$$
 $$+ \int_{\mathbb C^n} \sum_{j,k=1}^n p_j^* u_j \, \overline{p_k^*u_k}\, e^{-|z|^2}\,d\lambda$$
 $$=  \sum_{j,k=1}^n \int_{\mathbb C^n} |p_k(u_j)|^2\, e^{-|z|^2}\,d\lambda + 
  \sum_{j,k=1}^n \int_{\mathbb C^n} (p_j^* u_j \, \overline{p_k^*u_k} - p_k(u_j)\overline{p_j(u_k)})\, e^{-|z|^2}\, d\lambda$$
 $$=  \sum_{j,k=1}^n \int_{\mathbb C^n} |p_k(u_j)|^2\, e^{-|z|^2}\,d\lambda +  
  \sum_{j,k=1}^n \int_{\mathbb C^n} [p_k, p_j^*] u_j \overline{u_k}\, e^{-|z|^2}\,d\lambda,$$
where we used \eqref{gendef2}. Note that the expression 
$$\sum_{j,k=1}^n \int_{\mathbb C^n} |p_k(u_j)|^2\, e^{-|z|^2}\,d\lambda$$
is finite, since the components $u_j$ are polynomials, and it follows that the expression
$\sum_{j,k=1}^n ( [ p_k, p^*_j ] u_j, u_k)$ is a real number.

Now the assumption \eqref{comm1} implies that \eqref{basic6} holds for $(1,0)$-forms with polynomial components and, by Lemma \ref{dense}, we obtain \eqref{basic6} for any $u\in {\text{dom}}(D) \cap {\text{dom}}(D^* ).$
 \end{proof}
\begin{rem}In Theorem \ref{basic5} we implicitly suppose that  the expression 
$$\sum_{j,k=1}^n ( [ p_k, p^*_j ] u_j, u_k)$$
 is nonnegative.
\end{rem}

First we consider the one-dimensional case.
Let $p_m$ denote the polynomial differential operator
$$p_m = a_0 +a_1 \frac{\partial}{\partial z} + \dots + a_m \frac{\partial^m}{\partial z^m},$$
with constant coefficients $a_0,a_1, \dots, a_m \in \mathbb C,$ and let $p_m^*$ denote the polynomial
$$p_m^*(z) = \overline a_0 + \overline a_1 z + \dots + \overline a_m z^m,$$
with the complex conjugate coefficients $\overline a_0,\overline a_1, \dots, \overline a_m \in \mathbb C.$

 We consider the densely defined operator
\begin{equation}\label{gendef1}
Du = p_m(u)\, dz,
\end{equation}
where $u\in A^2(\mathbb C, e^{-|z|^2})$ and $p_m(u)\,dz$ is considered as a $(1,0)$-form.  

It is clear that $D^2=0,$ as all $(2,0)$-forms are zero if $n=1,$ and that we have
\begin{equation}\label{gendef20}
(Du, v) = (u, D^*v),
\end{equation}
where $u\in {\text{dom}}(D)= \{ u \in A^2(\mathbb{C}, e^{-|z|^2}): Du\in A^2_{(1,0)}(\mathbb{C}, e^{-|z|^2})\}$ and 
$$D^*v\, dz=  p^*_m v.$$

In the sequel we consider the generalized box operator
$$\tilde \Box_{D,1}:= DD^*$$
as a densely defined self-adjoint operator on $A_{(1,0)}^2(\mathbb C, e^{-|z|^2})$ with
 ${\text{dom}} (\tilde\Box_{D,1}) =\{ f\in  {\text{dom}}(D^* ) :   D^* f \in {\text{dom}}(D )\}.$

\begin{lem}\label{binom}
Let $u$ be an arbitrary polynomial. Then 

$\int_{\mathbb C} [p_m,p_m^*]u(z) \, \overline{u(z)}\, e^{-|z|^2}\, d\lambda (z) $
\begin{equation}\label{comm11}
=\sum_{\ell =1}^m \ell\, ! \int_{\mathbb C} \left | \sum_{k=\ell }^m \binom {k}{\ell } a_k u^{(k- \ell )}(z) \right |^2 \, e^{-|z|^2}\,
d\lambda (z). 
\end{equation}
\end{lem} 

\begin{proof} On the right hand side of \eqref{comm11} we have the integrand
\begin{equation}\label{comm2}
\sum_{\ell =1}^m \ell\, ! \sum_{k,j=\ell }^m \binom {k}{\ell } \binom {j}{\ell }a_k \overline a_j u^{(k- \ell )}
\overline u^{(j- \ell )} .
\end{equation}
For the left hand side of \eqref{comm11} we first compute
\begin{equation}\label{comm3}
[p_m,p_m^*]u = \sum_{k =0}^m a_k \frac{\partial^k}{\partial z^k} ( \sum_{j=0}^m \overline a_j z^j u) - ( \sum_{j=0}^m \overline a_j z^j )  ( \sum_{j=0}^m a_j  u^{(j)}).
\end{equation}
We use the Leibniz rule to get 
\begin{equation}\label{comm4}
 \frac{\partial^k}{\partial z^k} ( \sum_{j=0}^m \overline a_j z^j u) = \sum_{j=0}^m \overline a_j \frac{\partial^k}{\partial z^k} (z^ju)= \sum_{j=0}^m \overline a_j \sum_{\ell=0}^j \ell \, ! \binom {k}{\ell} \binom{j}{\ell} u^{(k -\ell)} z^{j-\ell}, 
\end{equation}
notice that $\binom {k}{\ell}=0,$ in the case  $k<\ell .$ Hence we obtain
\begin{equation}\label{comm5}
[p_m,p_m^*]u = \sum_{j,k=1}^m a_k \overline a_j \sum_{\ell =1}^j \ell \, !  \binom {k}{\ell} \binom{j}{\ell} u^{(k -\ell)} z^{j-\ell}.
\end{equation}
After integration we obtain 

$\int_{\mathbb C} [p_m,p_m^*]u(z) \, \overline{u(z)}\, e^{-|z|^2}\, d\lambda (z)$
\begin{equation}\label{comm6}
=\sum_{j,k=1}^m a_k \overline a_j \sum_{\ell =1}^j \ell \, !  \binom {k}{\ell} \binom{j}{\ell} \int_{\mathbb C} u^{(k -\ell)} \overline u^{(j-\ell)} \, e^{-|z|^2}\, d\lambda (z).
\end{equation}
Now it is easy to show that integration of \eqref{comm2} coincides with \eqref{comm6} and we are done.
\end{proof}

\begin{lem}\label{basic55} There exists a constant $C>0$ such that 
\begin{equation}\label{basic66}
\|u\| \le C  \|D^*u  \|,
\end{equation}
for each $ u  \in  {\text{dom}}(D^* ).$
\end{lem}

\begin{proof} First let $u $ be a polynomial and note that the last term in \eqref{comm11} equals 
$$m! \int_{\mathbb C} |a_m u(z)|^2 \, e^{-|z|^2}\, d\lambda (z) = m! \, |a_m|^2 \| u \|^2 $$
and all the other terms are non-negative, see Lemma \ref{binom}. Now we get that
$$\|D^*u\|^2= (p_m^*u, p_m^*u) = ( [ p_m, p^*_m ] u, u) + (p_m(u),p_m(u)) \ge \frac{1}{C} \, \|u\|^2,$$
where 
$$C= \frac{1}{m! \, |a_m|^2},$$
if we suppose that $a_m \neq 0,$
and we are done.
Finally apply Lemma \ref{dense} to obtain the desired result.

\end{proof}

\begin{thm}\label{genneumann}
Let $D$ be as in \eqref{gendef1}.  Then $\tilde\Box_{D,1}= D D^*$ has a bounded inverse
$$\tilde N_{D,1} :  A^2_{(1,0)}(\mathbb{C}, e^{-|z|^2}) \longrightarrow  {\text{dom}}(\tilde\Box_{D,1}).$$
If $\alpha \in A^2_{(1,0)}(\mathbb{C}, e^{-|z|^2}),$  then $u_0= D^* \tilde N_{D,1} \alpha$
is the canonical solution of $Du=\alpha,$ this means $Du_0=\alpha$ and $u_0 \in  ({\text{ker}} D)^\perp = {\text{im}}D^*,$ and
$\| D^*\tilde N_{D,1} \alpha \| \le C \|\alpha\|,$ for some constant $C>0$ independent of $\alpha.$
\end{thm}

\begin{proof}
Using Lemma \ref{basic55} we obtain that 
$$\tilde\Box_{D,1} :  {\text{dom}}(\tilde\Box_{D,1}) \longrightarrow A^2_{(1,0)}(\mathbb{C}, e^{-|z|^2})$$
is bijective and has the bounded inverse $\tilde N_{D,1},$ see \cite{Has20} Theorem 5.1. The rest follows from 
\cite{Has20} Theorem 5.2.
\end{proof}

\section{Commutators}

Let $A_j$ and $B_j, j=1,\dots, n$ be operators satisfying 
$$[A_j,A_k]=[B_j,B_k]=[A_j,B_k]=0 , j\neq k$$
and 
$$[A_j,B_j]=I, j=1,\dots, n.$$
Let $P$ and $Q$ be polynomials of $n$ variables and write $A=(A_1, \dots, A_n)$ and $B=(B_1, \dots, B_n).$
Then 
\begin{equation}\label{hamil}
Q(A)P(B)= \sum_{|\alpha|\ge 0} \frac{1}{\alpha !} P^{(\alpha )}(B) Q^{(\alpha)}(A),
\end{equation}
where $\alpha =(\alpha_1,\dots, \alpha_n)$ are multiindices and $|\alpha |=\alpha_1 +\dots + \alpha_n$ and
$\alpha ! = \alpha_1 ! \dots \alpha_n !,$ see \cite{Q}, \cite{Tr}.

The assumptions are satisfied, if one takes $A_j = \frac{\partial}{\partial z_j}$ and $B_j=z_j$ the multiplication operator.
 The inspiration for this comes from quantum mechanics, where the annihilation operator $A_j$ can be represented by the differentiation with respect to $z_j$ on $A^2(\mathbb C^n, e^{-|z|^2})$ and its adjoint, the creation operator $B_j,$ by the multiplication by $z_j,$ both operators being unbounded densely defined (see \cite{F1}, \cite{FY}). One can show that $A^2(\mathbb C^n, e^{-|z|^2})$ with this action of the $B_j$ and $A_j$ is an irreducible representation $M$ of the Heisenberg group; by the Stone-von Neumann theorem it is the only one up to unitary equivalence. Physically $M$ can be thought of as the Hilbert space of a harmonic oscillator with $n$ degrees of freedom and Hamiltonian operator
$$H =  \sum_{j=1}^n \frac{1}{2} (A_j B_j+B_jA_j).$$

\begin{rem}
If we apply \eqref{hamil} for the one-dimensional case of Lemma \ref{binom}, we get 
\begin{equation*}
\int_{\mathbb C} [p_m,p_m^*]u(z) \, \overline{u(z)}\, e^{-|z|^2}\, d\lambda (z) =
\sum_{\ell =1}^m \int_{\mathbb C} \left | p_m^{(\ell)}u (z) \right |^2 \, e^{-|z|^2}\,
d\lambda (z),
\end{equation*}
which coincides with \eqref{comm11}.
\end{rem}
In the following we consider $\mathbb C^2$ and choose $p_1$ and $p_2$ to be polynomials of degree $2$ in $2$ variables.

\begin{thm}\label{dim2}
Let $p_1,p_2$ be polynomials of degree $2.$ Suppose that
\begin{equation}\label{cond2} 
p_2^{(e_1)*}p_1^{(e_1)}= \pm \, p_1^{(e_2)*}p_2^{(e_2)}, \ p_1^{(e_1)*}p_2^{(e_1)}= \pm \, p_2^{(e_2)*}p_1^{(e_2)}, 
\end{equation}
where $(e_1)$ and $(e_2)$ denote the derivatives with respect to $z_1$ and $z_2$ respectively. In addition suppose that for all derivatives $(\alpha)$ of order $2$ we have
\begin{equation}\label{cond22}
p_j^{(\alpha)*}p_k^{(\alpha)} = \delta_{j,k}  \, c_{j,\alpha}   \ j,k=1,2,
\end{equation}
where 
$$C_1= \sum_{|\alpha|=2} \frac{1}{\alpha !}\, c_{1,\alpha} >0 \ {\text{and}} \ 
C_2= \sum_{|\alpha|=2} \frac{1}{\alpha !}\, c_{2,\alpha} >0.$$ 
Then 
 \begin{equation}\label{comm12}
 \|u\|^2 \le \frac{1}{\min (C_1,C_2)} \,  \sum_{j,k=1}^2 ( [ p_k, p^*_j ] u_j, u_k),
 \end{equation}
for any $(1,0)$-form $u= \sum_{j=1}^2 u_j dz_j $ with polynomial components.
\end{thm}

\begin{proof}
Using \eqref{hamil} we obtain
\begin{equation}\label{hamil2}
 ( [ p_k, p^*_j ] u_j, u_k) = \sum_{|\alpha|\ge 1} \frac{1}{\alpha !}\, ( p_j^{(\alpha) *}p_k^{(\alpha)} u_j, u_k).
\end{equation}
Now we use \eqref{cond2} and get for the first order derivatives 
\begin{align*}
\sum_{j,k=1}^2  [ & ( p_j^{(e_1) *}p_k^{(e_1)} u_j, u_k)  +( p_j^{(e_2) *}p_k^{(e_2)} u_j, u_k) ]\\
& = ( p_1^{(e_1) *}p_1^{(e_1)} u_1, u_1)+ ( p_2^{(e_1) *}p_2^{(e_1)} u_2, u_2)\\
& \pm \, ( p_2^{(e_2) *}p_1^{(e_2)} u_1, u_2) \pm \, ( p_1^{(e_2) *}p_2^{(e_2)} u_2, u_1)  \\
& + \, ( p_1^{(e_2) *}p_1^{(e_2)} u_1, u_1)+ ( p_2^{(e_2) *}p_2^{(e_2)} u_2, u_2)\\
& \pm \, ( p_2^{(e_1) *}p_1^{(e_1)} u_1, u_2) \, \pm \, ( p_1^{(e_1) *}p_2^{(e_1)} u_2, u_1)\\
& = (p_1^{(e_1)} u_1 \pm p_2^{(e_1)} u_2, p_1^{(e_1)} u_1 \pm p_2^{(e_1)} u_2)\\
& +(p_1^{(e_2)} u_1 \pm p_2^{(e_2)} u_2, p_1^{(e_2)} u_1 \pm p_2^{(e_2)} u_2)\\
& = \| p_1^{(e_1)} u_1 \pm p_2^{(e_1)} u_2 \|^2 + \| p_1^{(e_2)} u_1 \pm p_2^{(e_2)} u_2 \|^2 .
\end{align*}
For the second order derivatives we obtain
$$\sum_{j,k=1}^2 \sum_{|\alpha|=2 } \frac{1}{\alpha !}\, ( p_j^{(\alpha) *}p_k^{(\alpha)} u_j, u_k) = C_1\|u_1\|^2 +
C_2\|u_2\|^2.$$
Hence we get 

$\sum_{j,k=1}^2 ( [ p_k, p^*_j ] u_j, u_k)$
$$ = C_1\|u_1\|^2 + C_2\|u_2\|^2 + \| p_1^{(e_1)} u_1 \pm p_2^{(e_1)} u_2 \|^2 + \| p_1^{(e_2)} u_1 \pm p_2^{(e_2)} u_2 \|^2 ,$$
which gives \eqref{comm12}.

\end{proof}

In a similar way one shows
\begin{thm}\label{dim23}
Let $p_1,p_2$ be polynomials of degree $2.$ Suppose that
\begin{equation}\label{cond23} 
p_2^{(e_1)*}p_1^{(e_1)}= \pm \, p_1^{(e_1)*}p_2^{(e_1)}, \ p_1^{(e_2)*}p_2^{(e_2)}= \pm \, p_2^{(e_2)*}p_1^{(e_2)}, 
\end{equation}
where $(e_1)$ and $(e_2)$ denote the derivatives with respect to $z_1$ and $z_2$ respectively. In addition suppose that for all derivatives $(\alpha)$ of order $2$ we have
\begin{equation}\label{cond223}
p_j^{(\alpha)*}p_k^{(\alpha)} = \delta_{j,k}  \, c_{j,\alpha}   \ j,k=1,2,
\end{equation}
where 
$$C_1= \sum_{|\alpha|=2} \frac{1}{\alpha !}\, c_{1,\alpha} >0 \ {\text{and}} \ 
C_2= \sum_{|\alpha|=2} \frac{1}{\alpha !}\, c_{2,\alpha} >0.$$ 
Then 
 \begin{equation*}\label{comm42}
 \|u\|^2 \le \frac{1}{\min (C_1,C_2)} \,  \sum_{j,k=1}^2 ( [ p_k, p^*_j ] u_j, u_k),
 \end{equation*}
for any $(1,0)$-form $u= \sum_{j=1}^2 u_j dz_j $ with polynomial components.
\end{thm}

Finally we exhibit some examples, where conditions \eqref{cond2} and \eqref{cond22}, or 
\eqref{cond23} and \eqref{cond223} are checked. Examples (a) and (c) are taken from \cite{Has20}, where $ \sum_{j,k=1}^n ( [ p_k, p^*_j ] u_j, u_k)$ was directly computed.

\begin{ex}\label{nice}
a)  We take $p_1= \frac{\partial^2}{\partial z_1 \partial z_2}$ and $p_2=  \frac{\partial^2}{\partial z_1^2} +  \frac{\partial^2}{\partial z_2^2}.$ Then $p_1^*(z)= z_1z_2$ and $p_2^*(z)= z_1^2+z_2^2$ and we see that 
\eqref{cond2} and \eqref{cond22} are satisfied:
$$p_2^{(e_1)*}p_1^{(e_1)}=  p_1^{(e_2)*}p_2^{(e_2)}, \ p_1^{(e_1)*}p_2^{(e_1)}=  p_2^{(e_2)*}p_1^{(e_2)},$$
and  we obtain
  
$ \sum_{j,k=1}^2 ( [ p_k, p^*_j ] u_j, u_k)$
$$ = \int_{\mathbb C^2}  (|u_1|^2+ 4|u_2|^2
+ \left |\frac{\partial u_1}{\partial z_1}+ 2 \frac{\partial u_2}{\partial z_2}\right |^2 + \left |\frac{\partial u_1}{\partial z_2}
+ 2\frac{\partial u_2}{\partial z_1}\right |^2 )\, e^{-|z|^2}\, d\lambda,
$$
for $u= \sum_{j=1}^2 u_j dz_j $ with polynomial components.

b) Taking $p_1= i\frac{\partial^2}{\partial z_1 \partial z_2}$ and $p_2=  \frac{\partial^2}{\partial z_1^2} +  \frac{\partial^2}{\partial z_2^2}$ we have that $p_1^*(z)= -iz_1z_2$ and $p_2^*(z)= z_1^2+z_2^2$ and that 
\eqref{cond2} and \eqref{cond22} are satisfied:
$$p_2^{(e_1)*}p_1^{(e_1)}=  -p_1^{(e_2)*}p_2^{(e_2)}, \ p_1^{(e_1)*}p_2^{(e_1)}=  -p_2^{(e_2)*}p_1^{(e_2)},$$
and  we obtain
  
$ \sum_{j,k=1}^2 ( [ p_k, p^*_j ] u_j, u_k)$
$$ = \int_{\mathbb C^2}  (|u_1|^2+ 4|u_2|^2
+ \left |\frac{\partial u_1}{\partial z_1}+ 2i \frac{\partial u_2}{\partial z_2}\right |^2 + \left |\frac{\partial u_1}{\partial z_2}
+ 2i\frac{\partial u_2}{\partial z_1}\right |^2 )\, e^{-|z|^2}\, d\lambda,
$$
for $u= \sum_{j=1}^2 u_j dz_j $ with polynomial components.

c)  Let $ p_k = \frac{\partial^2}{\partial z_k^2}, k =1,2.$ Then $p^*_j(z)=z_j^2, j=1,2$ and we see that
\eqref{cond23} and \eqref{cond223} are satisfied and we have
\begin{eqnarray*}
 \sum_{j,k=1}^2 ( [ p_k, p^*_j ] u_j, u_k) &=& \sum_{j,k=1}^2 (2\delta_{j,k} u_j,u_k) + \sum_{j,k=1}^2 (4\delta_{jk} z_j \frac{\partial u_j}{\partial z_k},u_k)\\
 &=& 2\|u\|^2 + 4 \sum_{j=1}^2 \left \| \frac{\partial u_j}{\partial z_j} \right \|^2.
 \end{eqnarray*}

d) For $p_1= \frac{\partial^2}{\partial z_1^2}+ \frac{\partial}{\partial z_2} $ and $p_2=  \frac{\partial}{\partial z_1} +  \frac{\partial^2}{\partial z_2^2}$ we have $p_1^*(z)= z_1^2+z_2$ and $p_2^*(z)= z_1+z_2^2$ and we see that 
\eqref{cond2} and \eqref{cond23} are not satisfied. In particular,

\begin{eqnarray*} \sum_{j,k=1}^2 ( [ p_k, p^*_j ] u_j, u_k)
 &=& 3( \|u_1\|^2+\|u_2\|^2) + 4 (\| \frac{\partial u_1}{\partial z_1}\|^2 + \| \frac{\partial u_2}{\partial z_2}\|^2)\\
 &+& 2(u_1,z_2u_2)+ 2(z_1u_1,u_2)+ 2(u_2,z_1u_1) + 2(z_2u_2,u_1)
\end{eqnarray*}
for $u= \sum_{j=1}^2 u_j dz_j $ with polynomial components.

If we take $u_1(z_1,z_2)=z_2^n$ and $u_2(z_1,z_2)=-z_2^{n-1},$ where $n>7,$ we get
\begin{eqnarray*}
 \sum_{j,k=1}^2 ( [ p_k, p^*_j ] u_j, u_k) &=& \pi^2 [3 n! +3 (n-1)! +4(n-1)(n-2)! - 4 n!]\\
 &=&
(n-1)! \, \pi^2  \, (7-n) <0.
\end{eqnarray*}
\end{ex}

{\bf Acknowledgment.} The author thanks the referees for several useful suggestions.

\end{document}